\documentclass[12pt, twoside, reqno] {amsart}

\usepackage{amsfonts}
\usepackage{amssymb}
\usepackage{latexsym}
\usepackage{graphicx,graphics}
\usepackage{epsf,graphicx,latexsym%
}

\newcommand{\be} {\begin{eqnarray}}
\newcommand{\ee} {\end{eqnarray}}
\newcommand{\bep} {\begin{eqnarray*}}
\newcommand{\eep} {\end{eqnarray*}}

\newcommand {\s}{\mathop{\mathcal{S}}\nolimits}
\newcommand {\m}{\mathop{\mathcal{M}}\nolimits}

\newcommand {\Hol}{\mathop{\rm Hol}\nolimits}

\renewcommand {\Re}{\mathop{\rm Re}\nolimits}

\newcommand{\R}{{\mathbb R}}

\newcommand{\B}{{\mathbb B}}
\newcommand{\C}{{\mathbb C}}
\newcommand{\U}{{\mathbb U}}

\newcommand {\D}{\mathbb{D}}

\newtheorem{remar}{Remark}
\newtheorem{examp}{Example}
\newtheorem{defin}{Definition}
\newtheorem{theorem}{Theorem}

\newcommand{\rema}{\begin{remar}\rm}
\newcommand{\erema}{$\blacktriangleright$\end{remar}}

\newcommand{\exa}{\begin{examp}\rm}
\newcommand{\eexa}{$\blacktriangleright$\end{examp}}

\def\lwvec(#1 #2){\linewd 0.1
           \lvec(#1 #2)
           \linewd 0.05}

\makeatletter
\@namedef{subjclassname@2020}{\textup{2020} Mathematics Subject Classification}
\makeatother

\begin{document}

\title[Fekete--Szeg\"{o} problem for spirallike mappings]{Note on the Fekete--Szeg\"{o} problem for spirallike mappings in Banach spaces}

\begin{abstract}
In this note we present a remark on the paper ``On the coefficient inequalities for a class of holomorphic mappings associated with
spirallike mappings in several complex variables" by Y.~Lai and Q.~Xu \cite{LX} published recently in the journal {\it Results in Mathematics}.

We show that one of the theorems  in \cite{LX} concerning the finite-dimensional space $\C^n$ is a direct consequence of another one, so it does not need an independent proof.
Moreover, we prove that a sharp norm estimate on the Fekete--Szeg\"{o} functional over spirallike mappings in a general Banach space can be deduced from a result in \cite{LX}.
\end{abstract}

\author[M. Elin]{Mark Elin}

\address{Department of Mathematics,
         Ort Braude College,
         Karmiel 21982,
         Israel}

\email{mark$\_$elin@braude.ac.il}

\author[F. Jacobzon]{Fiana Jacobzon}

\address{Department of Mathematics,
         Ort Braude College,
         Karmiel 21982,
         Israel}

\email{fiana@braude.ac.il}


\maketitle

We mostly preserve the notations used in \cite{LX} to make the reading more convenient. Let  $\mathbb{U}$ be the open unit disk in the complex plane $\C$ and $\B$ be the open unit ball in a complex Banach space $X$ equipped with norm~$\|\cdot\|$.

Let $X^*$ be the dual space of $X$. For each $x \in X \setminus \{0\}$, we define
\begin{equation}\label{Tx-set}
T(x)=\left\{T_x \in X^*: \|T_x\|=1 \text{ and } T_x(x)=\|x\|\right\}.
\end{equation}
According to the Hahn-Banach theorem (see, for example, Theorem 3.2 in \cite{Rudin}), $T(x)$ is nonempty. 
Let $Y$ be a Banach space (possibly, different from~$X$). Denote by $\Hol(\B,Y)$ the set of all holomorphic mappings from $\B$ into $Y$. It is well known
that if $f \in \Hol(\B,Y)$, then
\begin{equation}\label{Tayl-infinite}
f(y) =\sum_{n=0}^{\infty} \frac{1}{n!} D^nf(x)\left[(y - x)^n\right]
\end{equation}
for all $y$ in some neighborhood of $x\in \B$, where $D^nf(x):\prod_{k=1}^{n}X \to Y$ is the $n^{th}$  Fr\'{e}chet derivative of $f$ at $x$ and
\begin{equation}\label{n-deriv}
D^nf(x)\left[(y - x)^n\right] = D^nf(x)[y - x,\ldots, y- x].
\end{equation}
It is well-known that $D^nf(x)$ is a bounded symmetric $n$-linear operator.

Recall that a holomorphic mapping $f : \B\to X$ is called biholomorphic if the inverse $f^{-1}$ exists and
is holomorphic on the image $f(\B)$.  A mapping $f \in \Hol(\B,X)$ is said to be locally biholomorphic
if for each $x \in \B$ there exists a bounded inverse for the Fr\'{e}chet derivative $Df(x)$.

Further we will use the following definitions.
\begin{defin}[see \cite{GHK2003}]\label{def-sparlike}
Suppose $|\beta|<\frac{\pi}{2}$.  A normalized locally biholomorphic mapping $f : \B\to X$ is called spirallike of type $\beta$ if
\begin{equation}\label{f-sparlike}
\Re \left(e^{-i \beta} T_x\left( \left(Df(x)\right)^{-1}f(x)\right)\right)>0, \quad \forall x \in \B\setminus\{0\} \text{ and } \ T_x \in T(x).
\end{equation}
\end{defin}

Let $\widehat{\s}_\beta(\B)$ be the class of spirallike mappings of type $\beta$ on $\B$. In the case $X = \C$, $\B =\U$, we denote  $\widehat{\s}_\beta:=\widehat{\s}_\beta(\U).$
Moreover, the relation \eqref{f-sparlike} is equivalent to $\Re \left(e^{-i \beta} \frac{zf'(z)}{f(z)}\right)>0$. Therefore, Definition~\ref{def-sparlike} is the standard one for spirallike functions of type $\beta$ on $\U$ (see, for example, \cite{GHK2003}).

\begin{defin}[cf. \cite{TCh2014, XL}]\label{def-Mg-class}
Suppose $|\beta|<\frac{\pi}{2}$. Let $g:\U \to \C$ be a biholomorphic function such
that $g(0)=1$, $\Re g(z)>0$, $z \in \D$, and $g_\beta(z)= \frac{\cos\!\beta g(z)+\!i\sin\beta}{e^{i \beta}}.$ Denote 
\begin{equation}\label{Mg-class}
\widehat{\m}_g(\beta)\!:=\!\left\{h \in \Hol (\B, X)\!:\! \frac{\|x\|}{T_x(h(x))}\! \in g_\beta(\U), x \in \B\!\setminus\{0\}, T_x \in T(x) \!\right\}\!.
\end{equation}
\end{defin}

In the special case when $X = \C$, $\B = \U$,  the class $\widehat{\m}_g(\beta)$ consists of holomorphic functions $h$ that satisfy
\[
\frac{z}{h(z)}\! \in g_\beta(\U).
\]

Note that the class $\widehat{\m}_g(0)$ was introduced in \cite{GHK2002, GKo} and studied in \cite{GHKK2017, GHK2018, HHK2006, HKK2021}.
Furthermore, if one takes $g(z)=\frac{1+z}{1-z}$, $z \in \U$, in Definition~\ref{def-Mg-class}, then the fact
$ \left(Df(x)\right)^{-1}f(x) \in \widehat{\m}_g(\beta)$ is equivalent to $f \in \widehat{\s}_\beta(\B)$.

To proceed, we assume that $g : \U \to\C$ satisfies the conditions of Definition~\ref{def-Mg-class} and the number $\beta$, $-\frac{\pi}{2}<\beta<\frac{\pi}{2}$, is fixed. Define the function $\psi_\beta : \C \to \R^+$ by
\[
\psi_\beta(\lambda):=\cos \beta \cdot \, \max\!\left\{1,\left|(1-2\lambda)\cos \beta g'(0)+\frac{e^{i \beta}}{2}\!\cdot\!\frac{g''(0)}{g'(0)}\right|\!\right\}.
\]

Y. Lai and Q. Xu in \cite{LX} studied so-called mappings of one-dimensional type (that is, such that $F(x)=f(x)x$, where $f\in\Hol(\B,\C)$, see, for example, \cite{ES2004}) and proved the following result:
\begin{theorem}[Theorem 3.2 in \cite{LX}]\label{lem3-2}
Let $g \in \Hol(\U,\C)$ satisfy the conditions of Definition~\ref{def-Mg-class}, $f \in \Hol(\B,\C)$, $f(0)=1$, $F(x)=xf(x)$ and suppose that $\left(DF(x)\right)^{-1}F(x) \in \widehat{M_g}(\beta)$. Then for every $\lambda \in \C$ we have
\begin{equation*}\label{FS-forTx}
 \left|\frac{T_x\left(D^3F(0)(x^3)\right)}{3!\|x\|^3}\!-\!\lambda \left(\frac{T_x\left(D^2F(0)(x^2)\right)}{2!\|x\|^2}\right)^2\right|\leq \frac12\ \psi_\beta(\lambda)|g'(0)|,
\end{equation*}
where $x \in \B\setminus\{0\}$ and $T_x \in T(x).$ The above estimate is sharp.
\end{theorem}

Further, the authors of \cite{LX}  prove independently Theorem~3.3 that asserts that in the case where $X=\C^n$  and is equipped with the norm $\|x\|:=\max\limits_{1\le j\le n}\left| x_j \right|$, the hypotheses of Theorem~\ref{lem3-2} imply that inequality~\eqref{res-thm2} below holds and is sharp.  We claim that the mentioned Theorem~3.3 in \cite{LX} is a direct consequence of Theorem~\ref{lem3-2}. Moreover, we show that for any Banach space $X$ the following assertion holds.

 \begin{theorem}\label{th3-3}
 Let $g \in \Hol(\U,\C)$ satisfy the conditions of Definition~\ref{def-Mg-class}, and mapping $F\in\Hol(\B,X)$ be of one-dimensional type such that $F(0)=0$, $DF(0)=I$ and $\left(DF(x)\right)^{-1}F(x) \in \widehat{M_g}(\beta)$. Then for every $\lambda \in \C$,
 \begin{equation}\label{res-thm2}
\left\|\frac{D^3F(0)(x^3)}{3!}\!-\!\lambda \frac12 D^2F(0)\!\!\left(\!x,\frac{D^2F(0)(x^2)}{2!}\!\right)\right\|\!\leq\!\frac{\|x\|^3}{2} \psi_\beta(\lambda)|g'(0)|,
\end{equation}
where $x \in \B \setminus \{0\}.$ The above estimate is sharp.
\end{theorem}
\begin{proof} Since $F$ is of one-dimensional type and $F(0)=0,\ DF(0)=I$, it can be represented as $F(x)=xf(x)$, where $f \in H(\B,\C)$ with $f(0)=1.$

In order to find the first Fr\'{e}chet derivative we calculate
\begin{eqnarray*}\label{DF}
  F(z+w)-F(z) &=& f(z+w)(z+w)-f(z)z \\
  &=& (f(z+w)-f(z))z+f(z+w)w \\
  &=&Df(z)[w]z+ f(z)w +o(\|w\|).
\end{eqnarray*}
This implies
\begin{eqnarray}\label{DF0}
\begin{array}{l}
(i) \ \ \ DF(z)[w]=Df(z)[w]z+ f(z)w,  \\
(ii) \ \  DF(0)[x]=f(0)x,  \\
(iii) \ T_x(DF(0)(x))=T_x(f(0)x)=f(0)\|x\|.
\end{array}
\end{eqnarray}
Turn now to the second Fr\'{e}chet derivative:
\begin{eqnarray*}\label{D2F}
&&DF(z+w_1)[w_2]-DF(z)[w_2]\\
&=& Df(z+w_1)[w_2](z+w_1)+f(z+w_1)w_2-Df(z)[w_2]z -f(z)w_2 \\
&=&(Df(z+w_1)-Df(z))[w_2]z+Df(z+w_1)[w_2]w_1+ (f(z+w_1)-f(z))w_2\\
&=&D^2f(z)[w_1,w_2]z+Df(z)[w_2]w_1+Df(z)[w_1]w_2 +o(\left\|w_1\right\|\cdot\left\|w_2\right\|).
\end{eqnarray*}
This leads to
\begin{eqnarray}\label{DF20}
\begin{array}{l}
(i)\ \ \ D^2F(z)[w_1,w_2]=D^2f(z)[w_1,w_2]z+Df(z)[w_2]w_1+Df(z)[w_1]w_2,\\
(ii) \ \  D^2F(0)[w_1,w_2]=Df(0)[w_2]w_1+Df(0)[w_1]w_2,  \\
(iii) \  D^2F(0)[x^2]=2Df(0)[x]x,\\
(iv) \ T_x(D^2F(0)[x^2])=2Df(0)[x]T_x(x)=2Df(0)[x]\|x\|.
\end{array}
\end{eqnarray}

Using these calculations we can similarly find the third  Fr\'{e}chet derivative and get:
\begin{equation}\label{DF30}
 D^3\!F(0)\left[x^3\right]\!=\!3D^2\!f(0)\left[x^2\right]x\ \text{ and }\ T_x(D^3\!F(0)[x^3])\!=\!3D^2\!f(0)[x^2]\|x\|.
\end{equation}

We now apply formulae \eqref{DF20}(iii) and \eqref{DF30} to compute the analog of the Fekete--Szeg\"{o} functional for $F$:
\begin{eqnarray*}\label{Thm3-3inRIMA-calc}
&&\frac{D^3F(0)[x^3]}{3!}-\lambda \frac12 D^2F(0)\left[x,\frac{D^2F(0)(x^2)}{2!}\right]\\
&=& \frac{D^2f(0)[x^2]}{2!}x-\lambda \frac12 D^2F(0)\left[x,Df(0)(x)x\right]\\
&=& \frac{D^2f(0)[x^2]}{2!}x-\lambda \frac12 \left(Df(0)[Df(0)[x]x]x+Df(0)[x]Df(0)[x]x\right)\\
&=&\!\!\frac{D^2f(0)[x^2]}{2!}x-\! \lambda  Df(0)[x] Df(0)[x]x\\
&=& \left(\!\frac{D^2f(0)[x^2]}{2!}-\lambda  \left(Df(0)[x]\right)^2\! \right)\!x.
\end{eqnarray*}
Thus this is a mapping of one-dimensional type. Combining formulae \eqref{DF0}--\eqref{DF30} we have
\begin{eqnarray*}\label{Thm3-3inRIMA-new}
&&\left\|\frac{D^3F(0)[x^3]}{3!}-\lambda \frac12 D^2F(0)\left[x,\frac{D^2F(0)[x^2]}{2!}\right]\right\|\\
&=& \left|\frac{D^2f(0)[x^2]}{2!}-\lambda  \left(Df(0)[x]\right)^2 \right|\cdot \|x\|\\
&=& \|x\|^3 \left|\frac{T_x\left(D^3F(0)[x^3]\right)}{3!\|x\|^3}-\lambda \left(\frac{T_x(D^2F(0)[x^2])}{2!\|x\|}\right)^2 \right|.
\end{eqnarray*}
This equality together with Theorem~\ref{lem3-2} implies that estimate \eqref{res-thm2} holds and is sharp. 
\end{proof}

\end{document}